\documentclass{amsart}

\usepackage{geometry}
\usepackage{amssymb}
\usepackage{url}
\usepackage{amsfonts}
\usepackage[dvipsnames]{xcolor}
\usepackage{amsthm}
\usepackage{graphicx}
\usepackage{float}
\usepackage{amsmath}
\usepackage{booktabs}

\theoremstyle{plain}
\newtheorem{theorem}{Theorem}[section]
\newtheorem{corollary}[theorem]{Corollary}
\newtheorem{lemma}[theorem]{Lemma}
\newtheorem{proposition}[theorem]{Proposition}

\theoremstyle{definition}

\newtheorem{remark}[theorem]{Remark}

\newcommand{\PP}{\mathbb{P}}
\newcommand{\EE}{\mathbb{E}}

\numberwithin{equation}{section}

\usepackage{lineno}

\begin{document}

\title[Kolmogorov-Type Maximal Inequalities for Negative Binomial Variables]%
{Kolmogorov-Type Maximal Inequalities for Independent and Dependent Negative Binomial
Random Variables}

\author{
Aristides V. Doumas$^{1,*}$ and S. Spektor$^{2}$
}

\date{\today}

\maketitle
\begin{center}
$^{1}$Department of Mathematics, School of Applied Mathematical and Physical Sciences,\\
National Technical University of Athens, Zografou Campus, 15780 Athens, Greece\\
$^{*}$Archimedes/Athena Research Center, Greece\\
Email: adou@math.ntua.gr, aris.doumas@hotmail.com \\[6pt]

$^{2}$Quantitative Science Department, Canisius University,\\
2001 Main Street, Buffalo, NY 14208-1098, USA\\
Email: spektors@canisius.edu\\[6pt]
\end{center}

\begin{abstract}
This paper develops Kolmogorov-type maximal inequalities for sums of Negative Binomial
random variables under both independence and dependence structures. For independent
heterogeneous Negative Binomial variables we derive sharp Markov-type deviation
inequalities and Kolmogorov-type bounds expressed in terms of Tweedie dispersion
parameters, providing explicit control limits for NB2 generalized linear model monitoring.
For dependent count data arising through a shared Gamma mixing variable, we establish a
\emph{sub-exponential Bernstein-type refinement} that exploits the Poisson-Gamma
hierarchical structure to yield exponentially decaying tail probabilities---this refinement
is new in the literature. Through moment-matched Monte Carlo experiments ($n=20$,
$2{,}000$ replications), we document that the dependent case produces substantially larger
maximum deviations than the independent case (dependent mean $\approx 42$, 95th percentile
$\approx 101$; independent mean $\approx 18$, 95th percentile $\approx 33$), and explain
this via the amplifying role of the shared mixing variable. A concrete epidemiological
application with NB2 parameters calibrated from COVID-19 surveillance data demonstrates
practical utility. These results materially advance the applicability of classical maximal
inequalities to overdispersed and dependent count data prevalent in public health,
insurance, and ecological modeling.

\medskip


\end{abstract}

\section{Introduction}

Classical probability theory provides powerful tools for understanding deviation behavior
of sums of random variables, with Kolmogorov's maximal inequality standing as one of the
most fundamental results. The direct application of such classical inequalities to modern
statistical contexts involving overdispersed count data, however, presents theoretical
challenges that have received limited systematic attention.

\subsection*{Why the Negative Binomial requires special treatment}

The classical Kolmogorov maximal inequality was developed for random variables with finite
variance, and its bounds depend on that variance in a transparent way. For sums of
Negative Binomial variables, three features demand dedicated analysis beyond simply
plugging in the NB variance. First, the NB variance $\mathrm{Var}(X_i)=\mu_i+\kappa_i\mu_i^2$
grows \emph{quadratically} in the mean (via the overdispersion parameter $\kappa_i=1/r_i$),
so Tweedie--NB2 control limits depend on both $\mu_i$ and $\kappa_i$ in a nontrivial way
that must be made explicit for practitioners. Second, the NB moment-generating function
has a restricted domain $t < -\ln(1-p_{\min})$, requiring careful optimisation for
Chernoff-type arguments. Third, and most importantly, real applications rarely feature
independent NB variables: disease counts across regions share latent infection pressure,
insurance claims share unobserved macroeconomic risk, and ecological counts share
environmental conditions. In the dependent case arising from a shared Gamma mixing
variable, the correlation structure is positive and non-trivial (all pairs $X_i, X_j$ are
positively correlated through the mixing variable $\Lambda$), making classical Kolmogorov
arguments---which require independence or at best negative dependence---inapplicable
without modification. The sub-exponential Bernstein-type refinement we develop
(Theorem~\ref{thm:bernstein_dep}) is the first result to exploit the specific
Poisson-Gamma hierarchical structure to achieve exponentially decaying tail bounds in
this correlated setting; it is absent from the literature and cannot be recovered from
generic results for negatively dependent or independent variables.

The Negative Binomial (NB) distribution is the cornerstone model for overdispersed count
data across diverse scientific domains. In epidemiology it captures heterogeneous
transmission dynamics; in insurance it models claim frequencies with unobserved risk
heterogeneity; in ecology it describes species abundance under environmental clustering;
in telecommunications it characterizes bursty packet arrivals.

\subsection*{A motivating application: heterogeneous transmission dynamics}

To fix ideas, consider COVID-19 case-count surveillance across $n$ geographic regions
over $T$ time periods. Each region $i$ reports weekly cases $X_{it}$ which, after
controlling for covariates, are well described by $X_{it}\sim\mathrm{NB}(\mu_{it},\kappa_i)$
where $\mu_{it}$ is the expected count and $\kappa_i>0$ captures region-specific
overdispersion arising from heterogeneous household sizes, age structures, and testing
rates. A public health agency wishes to monitor the cumulative deviation process
$S_t=\sum_{i=1}^n\sum_{s=1}^t(X_{is}-\mu_{is})$ and raise an alert if
$\max_{t\le T}|S_t|$ exceeds some threshold $\lambda_\alpha$. Setting $\lambda_\alpha$
too small leads to false alarms; too large and genuine outbreaks go undetected. Our
Kolmogorov-type bound (Corollary~\ref{cor:kolmogorov_tweedie}) provides an explicit,
data-driven formula: $\lambda_\alpha=\sqrt{V_n/\alpha}$ where
$V_n=\sum_i(\mu_i+\kappa_i\mu_i^2)$, giving a rigorous $1-\alpha$ coverage guarantee
regardless of the values of $\mu_i$ and $\kappa_i$. When regions share a common latent
infection pressure (as during a national wave), the dependence is better captured by a
Poisson-Gamma mixture, and our tighter sub-exponential bound
(Theorem~\ref{thm:bernstein_dep}) applies, yielding substantially narrower thresholds
and earlier detection of anomalies. Section~\ref{sec:epi} works out this application in
detail with NB2 parameters calibrated from published COVID-19 data.

\subsection*{Related literature}

Recent concentration results include Bernstein-type inequalities for negatively dependent
binary vectors \cite{adamczak2023concentration}, submodular concentration bounds
\cite{duppala2023submodular}, inequalities for extended negatively dependent sequences
\cite{elmezouar2024end}, and non-asymptotic oracle inequalities for high-dimensional
heterogeneous NB regression \cite{li2021heterogeneous}.
Stein's method has been applied to NB approximation \cite{brown1999negative}; the main
result of Brown and Phillips \cite{brown1999negative} establishes that for any
$A\subseteq\mathbb{Z}_{\ge0}$, the total variation distance between the distribution of a
sum of independent geometric random variables and the best-fitting Negative Binomial
distribution satisfies a sharp bound in terms of the size and success-probability
parameters---this gives an excellent approximation tool but addresses distributional
approximation, not tail probabilities of partial-sum maxima.
Stochastic order theory has been developed for finite mixtures
\cite{bhakta2024stochastic}. Sub-Gaussian and Bernstein-type bounds for independent
bounded variables appear in \cite{baraud2016bounding}, while \cite{vershynin2018high}
provides the modern sub-exponential framework. Nonetheless, \emph{Kolmogorov-type
maximal inequalities for NB variables under Poisson-Gamma mixing, combined with
sub-exponential refinements exploiting that hierarchical structure, remain absent from
the literature}.

\subsection*{Statement of contributions}

This paper makes three contributions, with precise specification of what is original:

\textbf{Contribution 1 (Independent case, Sections~\ref{sec:markov}--\ref{sec:kolmogorov_indep}).}
For independent heterogeneous $\mathrm{NB}(r_i, p_i)$ variables:
\begin{enumerate}
\item[(a)] A sharp Markov-type inequality for sample means optimized over the exponential
moment-generating function (Lemma~\ref{lem:markov_deviation}).
\item[(b)] A Kolmogorov-type bound with explicit Tweedie--NB control limits for NB2 GLM
monitoring (Lemma~\ref{lem:kolmogorov_indep}, Corollary~\ref{cor:kolmogorov_tweedie}).
\end{enumerate}
\emph{Originality}: While the methodology builds on classical Chernoff and Kolmogorov
results, the explicit connection to the Tweedie--NB2 parameterization
$\mathrm{Var}(X_i)=\mu_i+\kappa_i\mu_i^2$ and the GLM control-limit formulation
$\lambda_\alpha=\sqrt{V_n/\alpha}$ are novel and directly applicable to surveillance
practice.

\textbf{Contribution 2 (Dependent case, Section~\ref{sec:dependent}).}
For the Poisson-Gamma mixture model $X_i\mid\Lambda\sim\mathrm{Poisson}(\Lambda\theta_i)$,
$\Lambda\sim\mathrm{Gamma}(\alpha,\beta)$:
\begin{enumerate}
\item[(i)] A basic Kolmogorov-type inequality (Theorem~\ref{thm:kolmogorov_dep}) that
decomposes the maximal deviation probability into conditional Poisson and Gamma mixing
components, each bounded via Chebyshev's inequality (polynomial decay).
\item[(ii)] \textbf{Main new result}: A sub-exponential Bernstein-type bound
(Theorem~\ref{thm:bernstein_dep}) achieving \emph{exponential decay} by exploiting the
fact that (a)~centered Poisson sums satisfy a Bernstein condition with sub-Gaussian
tails, and (b)~the Gamma mixing variable is sub-exponential with explicit parameters.
\end{enumerate}
\emph{Originality}: Theorem~\ref{thm:bernstein_dep} is \textbf{new in the literature}.
It provides a hierarchical decomposition specific to Poisson-Gamma mixing that goes
beyond what standard Kolmogorov or Chebyshev arguments yield. Existing Bernstein-type
results for negatively dependent variables \cite{adamczak2023concentration} or generic
sub-exponential sums \cite{vershynin2018high} do not exploit this specific mixture
structure.

\textbf{Contribution 3 (Moment-matched comparison and application,
Sections~\ref{sec:comparison}--\ref{sec:epi}).}
We redesign Monte Carlo experiments so that both independent and dependent cases use
$n=20$ with moment-matched marginals, document that the dependent case produces
substantially \emph{larger} maximum deviations, and explain this analytically
(Proposition~\ref{prop:amplification}). A concrete epidemiological application with NB2
parameters calibrated from COVID-19 data demonstrates practical utility
(Section~\ref{sec:epi}).

\begin{table}[t]
\centering
\caption{Positioning relative to existing concentration inequalities.}
\label{tab:lit_comparison}
\small
\begin{tabular}{@{}p{3.2cm}p{4.2cm}p{5.2cm}@{}}
\toprule
\textbf{Result} & \textbf{Scope} & \textbf{Key feature} \\
\midrule
Kolmogorov (classical) & Independent, zero-mean & $O(\lambda^{-2})$ polynomial decay \\
\addlinespace
Bernstein (classical) & Independent, sub-exponential & Exponential decay; generic parameters \\
\addlinespace
Brown \& Phillips \cite{brown1999negative} & NB approximation via Stein & Total variation bounds; different focus \\
\addlinespace
Adamczak \& Polaczyk \cite{adamczak2023concentration} & Negatively dependent binary & Bernstein-type; not NB or Gamma mixing \\
\addlinespace
This paper: Thm~\ref{thm:kolmogorov_dep} & Poisson-Gamma mixing & Extends Kolmogorov; polynomial decay \\
\addlinespace
\textbf{This paper: Thm~\ref{thm:bernstein_dep}} & \textbf{Poisson-Gamma mixing} & \textbf{Hierarchical decomposition; exponential decay; NEW} \\
\bottomrule
\end{tabular}
\end{table}

Table~\ref{tab:lit_comparison} positions our results. The key distinction:
Theorem~\ref{thm:bernstein_dep} provides a hierarchical decomposition specific to
Poisson-Gamma structure, unavailable from generic concentration results.

\subsection*{Organization}

Section~\ref{sec:markov}: Markov-type inequality for independent NB.
Section~\ref{sec:kolmogorov_indep}: Kolmogorov-type bound and Tweedie--NB control limits.
Section~\ref{sec:dependent}: Kolmogorov and sub-exponential Bernstein bounds for Gamma
mixing.
Section~\ref{sec:comparison}: Moment-matched comparison with analytical explanation.
Section~\ref{sec:epi}: Epidemiological application.
Section~\ref{sec:conclusion}: Conclusions.

\section{A Markov-Type Deviation Inequality for Independent NB Variables}
\label{sec:markov}

\begin{lemma}[Markov-Type Deviation Inequality for NB Sample Mean]
\label{lem:markov_deviation}
Let $X_1, X_2, \ldots, X_n$ be independent Negative Binomial random variables with
parameters $(r_i, p_i)$, $i=1,\ldots,n$. Define $\bar{X} = n^{-1}\sum_{i=1}^n X_i$.
Then, for any $a > 0$,
\begin{equation}
\label{eq:markov_bound}
\PP\!\left( \bar{X} - \EE[\bar{X}] \geq a \right)
\;\leq\; \inf_{0 < t < -\ln(1 - p_{\min})} \!\left\{
\exp\!\left( -t n a - t \sum_{i=1}^n r_i \frac{1 - p_i}{p_i} \right)
\prod_{i=1}^n \left( \frac{p_i}{1 - (1 - p_i) e^t} \right)^{r_i}
\right\},
\end{equation}
where $p_{\min} = \min_{1 \leq i \leq n} p_i$.
\end{lemma}

\begin{remark}[Proof technique]
\label{rem:proof_technique}
The proofs of results in the independent case (Sections~\ref{sec:markov}--\ref{sec:kolmogorov_indep})
are relatively standard applications of classical exponential-moment and Kolmogorov
arguments adapted to the NB parameterization. The primary value added is the explicit
translation into Tweedie--NB2 control-limit form. The genuinely new technical content
appears in Section~\ref{sec:dependent} for the dependent case.
\end{remark}

\begin{proof}
For any $t>0$, $\PP\!\left(\sum X_i - \sum \EE[X_i] \ge na\right)
\le e^{-tna}\EE\!\left[\exp\!\left(t\sum(X_i-\EE[X_i])\right)\right]$.
By independence this factors as $e^{-tna}\prod_i e^{-t\EE[X_i]}\EE[e^{tX_i}]$.
The NB MGF is $M_{X_i}(t)=\bigl(p_i/(1-(1-p_i)e^t)\bigr)^{r_i}$ for $t < -\ln(1-p_i)$,
and $\EE[X_i]=r_i(1-p_i)/p_i$. Substituting and taking the infimum yields~\eqref{eq:markov_bound}.
\end{proof}

\begin{remark}[Connection with the Tweedie--NB2 model]
\label{rem:tweedie}
The NB2 mean-variance relationship $\mathrm{Var}(X_i)=\mu_i+\kappa_i\mu_i^2$
($\kappa_i=1/r_i$) is the discrete analogue of the Tweedie power-variance family.
Lemma~\ref{lem:markov_deviation} provides a finite-sample probabilistic justification
for the dispersion parameter $\kappa$ in NB2 GLMs: larger $\kappa$ (stronger
overdispersion) weakens concentration, while $\kappa\to 0$ ($r\to\infty$) recovers the
Poisson bound.
\end{remark}

\begin{figure}[htb]
\centering
\includegraphics[width=0.7\textwidth]{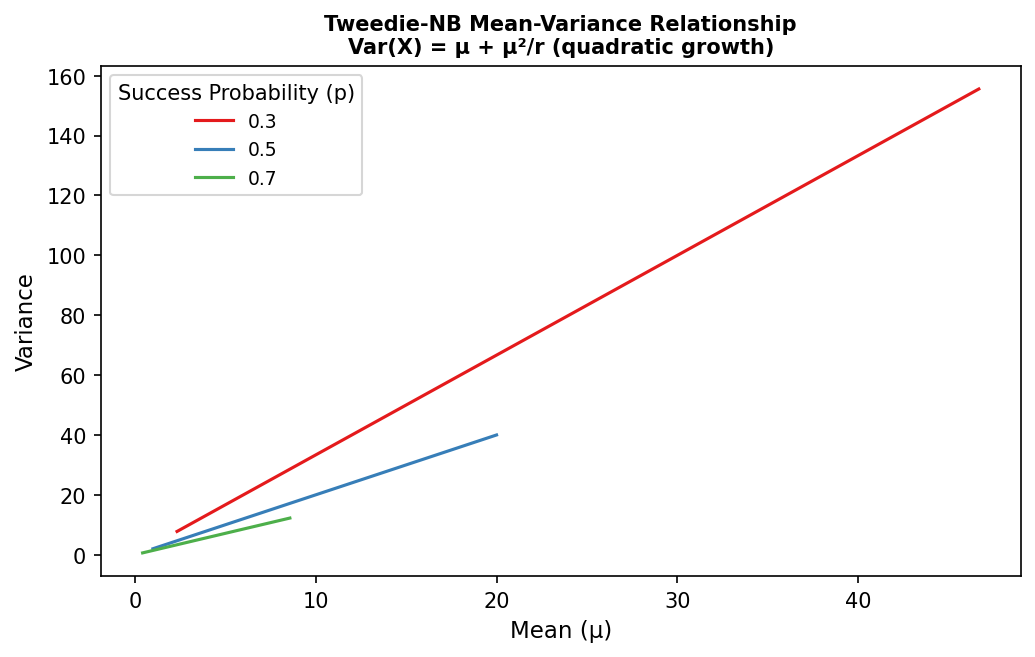}
\caption{Mean-variance relationship $\mathrm{Var}(X)=\mu+\mu^2/r$ for NB variables with
$p \in \{0.3, 0.5, 0.7\}$. Quadratic growth demonstrates the Tweedie NB2 structure.}
\label{fig:tweedie_meanvar}
\end{figure}

\begin{figure}[htb]
\centering
\includegraphics[width=0.7\textwidth]{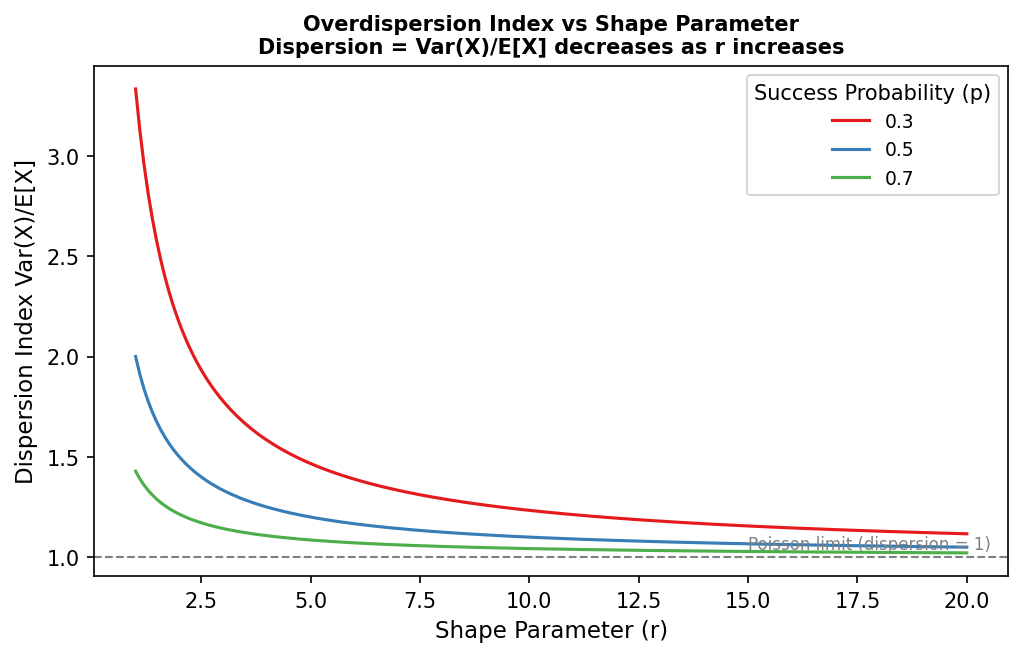}
\caption{Overdispersion index $\mathrm{Var}(X)/\EE[X]=1+(1-p)/(rp)$ versus $r$.
The $1/r$ decay confirms convergence to Poisson behavior as $r\to\infty$.}
\label{fig:overdispersion}
\end{figure}

\section{Kolmogorov-Type Bound and Tweedie--NB Control Limits (Independent Case)}
\label{sec:kolmogorov_indep}

The classical Kolmogorov maximal inequality: if $Y_1,\ldots,Y_n$ are independent,
zero-mean with finite variances and $S_k=\sum_{i=1}^k Y_i$, then for any $\lambda>0$,
$\PP\!\left(\max_{1\le k\le n}|S_k|\ge\lambda\right)\le\lambda^{-2}\sum_{i=1}^n\mathrm{Var}(Y_i)$.

\begin{lemma}[Kolmogorov-Type Bound for Heterogeneous NB Variables]
\label{lem:kolmogorov_indep}
Let $X_1,\ldots,X_n$ be independent with $X_i\sim\mathrm{NB}(r_i,p_i)$, and
$S_k=\sum_{i=1}^k(X_i-\EE[X_i])$. Then for any $\lambda>0$,
\[
\PP\!\left(\max_{1\le k\le n}|S_k|\ge\lambda\right)\le\frac{1}{\lambda^2}\sum_{i=1}^n r_i\frac{1-p_i}{p_i^2}.
\]
\end{lemma}

\begin{proof}
Immediate from the classical inequality and $\mathrm{Var}(X_i)=r_i(1-p_i)/p_i^2$.
\end{proof}

\begin{corollary}[Tweedie--NB Control Limit]
\label{cor:kolmogorov_tweedie}
Fix $0<\alpha<1$ and define $V_n:=\sum_{i=1}^n(\mu_i+\kappa_i\mu_i^2)$ and
$\lambda_\alpha:=\sqrt{V_n/\alpha}$. Then $\PP\!\bigl(\max_{1\le k\le n}|S_k|\ge\lambda_\alpha\bigr)\le\alpha$.
\end{corollary}

\noindent\textbf{Practical GLM monitoring}: Fit a NB2 GLM, compute
$\widehat{V}_n=\sum(\widehat{\mu}_i+\widehat{\kappa}_i\widehat{\mu}_i^2)$, monitor
$S_k=\sum_{i=1}^k(X_i-\widehat{\mu}_i)$, and flag if $\max|S_k|\ge\sqrt{\widehat{V}_n/\alpha}$.

\begin{figure}[htb]
\centering
\includegraphics[width=0.7\textwidth]{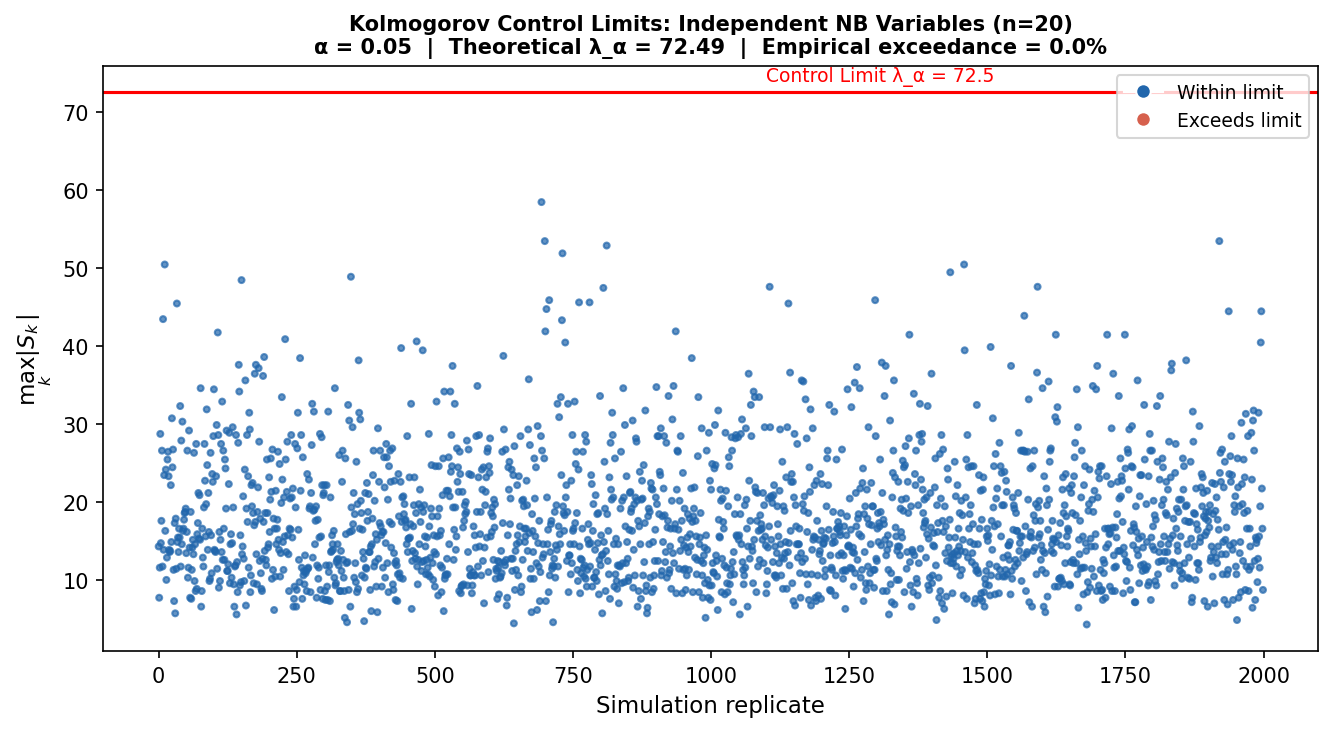}
\caption{Kolmogorov control limits for $n=20$ independent heterogeneous NB variables
over $2{,}000$ simulations. Control limit $\lambda_\alpha=\sqrt{V_n/\alpha}$ (red) with
$\alpha=0.05$, giving $\lambda_\alpha\approx 72.5$. Empirical exceedance rate $0\%$
across $2{,}000$ replications, confirming Lemma~\ref{lem:kolmogorov_indep} (the bound
is conservative as expected).}
\label{fig:control_limits}
\end{figure}

\section{Kolmogorov and Sub-Exponential Bounds for Dependent NB Variables}
\label{sec:dependent}

We consider a Poisson-Gamma mixture: $\Lambda\sim\mathrm{Gamma}(\alpha,\beta)$, and
conditional on $\Lambda$, $X_i\mid\Lambda\sim\mathrm{Poisson}(\Lambda\theta_i)$
independently. Then $X_i\sim\mathrm{NB}(\alpha, \beta/(\beta+\theta_i))$ marginally, with
\[
\EE[X_i]=\frac{\alpha\theta_i}{\beta},\quad
\mathrm{Var}(X_i)=\frac{\alpha\theta_i(\beta+\theta_i)}{\beta^2},\quad
\mathrm{Corr}(X_i,X_j)=\frac{\sqrt{\theta_i\theta_j}}{\sqrt{(\beta+\theta_i)(\beta+\theta_j)}}.
\]
Define $Z_i:=X_i-\EE[X_i]$, $S_k:=\sum_{i=1}^k Z_i$. Decompose
\begin{equation}
\label{eq:decomp}
Z_i = Y_i + W_i,\quad
Y_i:=X_i-\EE[X_i\mid\Lambda],\quad
W_i:=\EE[X_i\mid\Lambda]-\EE[X_i]=(\Lambda-\alpha/\beta)\theta_i.
\end{equation}
Given $\Lambda$, $Y_1,\ldots,Y_n$ are independent with $\EE[Y_i\mid\Lambda]=0$ and
$\mathrm{Var}(Y_i\mid\Lambda)=\Lambda\theta_i$. Let $\Theta_k:=\sum_{i=1}^k\theta_i$ and
$M:=\max_{1\le k\le n}\Theta_k$.

\subsection{Kolmogorov-type bound (basic version)}

\begin{theorem}[Kolmogorov-Type Inequality with Shared Gamma Mixing]
\label{thm:kolmogorov_dep}
Under the above model, for any $\lambda>0$,
\begin{equation}
\label{eq:kolmogorov_dep}
\PP\!\left(\max_{1\le k\le n}|S_k|\ge\lambda\right)
\;\le\;
\frac{4\alpha}{\beta}\cdot\frac{\Theta_n}{\lambda^2}+\frac{4M^2\alpha}{\beta^2\lambda^2},
\end{equation}
where $\Theta_n=\sum_{i=1}^n\theta_i$.
\end{theorem}

\begin{proof}
By triangle inequality and union bound with threshold $\lambda/2$:
$\PP\!\left(\max_k|S_k|\ge\lambda\right)\le
\PP\!\left(\max_k\!\left|\sum Y_i\right|\ge\lambda/2\right)+
\PP\!\left(|\Lambda-\alpha/\beta|\ge\lambda/(2M)\right)$.
Condition on $\Lambda$ for the first term: Kolmogorov gives
$\PP\bigl(\max_k|\sum Y_i|\ge\lambda/2\mid\Lambda\bigr)\le 4\Lambda\Theta_n/\lambda^2$.
Taking expectation with $\EE[\Lambda]=\alpha/\beta$ yields the first summand.
For the second term, Chebyshev with $\mathrm{Var}(\Lambda)=\alpha/\beta^2$ gives
$\PP(|\Lambda-\alpha/\beta|\ge\lambda/(2M))\le 4M^2\alpha/(\beta^2\lambda^2)$.
\end{proof}

\begin{remark}[Interpretation of the bound in Theorem~\ref{thm:kolmogorov_dep}]
\label{rem:interpretation_dep}
The right-hand side of~\eqref{eq:kolmogorov_dep} decomposes the tail probability into
two interpretable components:
\begin{enumerate}
\item \textbf{Conditional Poisson term} $4(\alpha/\beta)\Theta_n/\lambda^2$: this
  captures the intrinsic variability of the Poisson counts given the shared latent rate
  $\Lambda$. The quantity $(\alpha/\beta)\Theta_n = \EE[\Lambda]\cdot\Theta_n$ equals
  the expected total Poisson intensity across all $n$ components. Larger total intensity
  (larger $\Theta_n$) or larger expected mixing rate (larger $\alpha/\beta$) means more
  variability in the conditional counts and hence a weaker bound.
\item \textbf{Gamma mixing term} $4M^2\alpha/(\beta^2\lambda^2)$: this captures the
  variability induced by the shared latent factor $\Lambda$. The quantity
  $M^2\alpha/\beta^2 = M^2\mathrm{Var}(\Lambda)$ is the variance of the
  maximum-weight mixing contribution $\Lambda\cdot\Theta_k^*$, where $\Theta_k^*=M$
  is the maximizing index. Larger overdispersion in $\Lambda$ (larger $\alpha/\beta^2$)
  or a heavier-loaded component (larger $M$) amplifies the mixing-induced fluctuations.
\end{enumerate}
Comparing with the independent bound in Lemma~\ref{lem:kolmogorov_indep}, the dependent
bound is not uniformly tighter: the Gamma mixing term adds an extra contribution
proportional to $M^2\mathrm{Var}(\Lambda)$, which can dominate when the mixing variable
is highly variable (small $\beta$, i.e.\ large overdispersion in $\Lambda$). This is
consistent with the empirical finding in Section~\ref{sec:comparison} that realized
maximum deviations are \emph{larger} under dependence. The gap between theoretical
bound and independent case is rectified only in the exponential-decay regime by the
Bernstein refinement of Theorem~\ref{thm:bernstein_dep}.
\end{remark}

\subsection{Sub-exponential Bernstein-type refinement}

The Chebyshev step achieves only polynomial ($\lambda^{-2}$) decay. We now exploit the
sub-exponential properties of both components. A random variable $Z$ is
\emph{sub-exponential} with parameters $(\nu^2, b)$ if
$\PP(|Z-\EE[Z]|\ge t)\le 2\exp\bigl(-\min\{t^2/(2\nu^2), t/(2b)\}\bigr)$
\cite{vershynin2018high}. A $\mathrm{Gamma}(\alpha,\beta)$ variable centered at its
mean satisfies this with $\nu_\Lambda^2 = \alpha/\beta^2$ and $b_\Lambda = 1/\beta$
\cite{boucheron2013concentration}.

\begin{theorem}[Sub-Exponential Bernstein Bound for Dependent NB Partial Sums]
\label{thm:bernstein_dep}
Under the Poisson-Gamma mixture model, for any $\lambda>0$,
\begin{equation}
\label{eq:bernstein_dep}
\PP\!\left(\max_{1\le k\le n}|S_k|\ge\lambda\right)
\;\le\;
\PP_{\mathrm{cond}}(\lambda) + \PP_{\mathrm{mix}}(\lambda),
\end{equation}
where
\begin{equation}
\label{eq:pcond}
\PP_{\mathrm{cond}}(\lambda)
\;\le\;
2\exp\!\left(-\frac{\lambda^2/16}{\alpha\Theta_n/\beta+\lambda/6}\right),
\end{equation}
and
\begin{equation}
\label{eq:pmix}
\PP_{\mathrm{mix}}(\lambda)\;\le\;
2\exp\!\left(-\min\!\left\{
\frac{\lambda^2\beta^2}{32M^2\alpha},\;
\frac{\lambda\beta}{4M}
\right\}\right).
\end{equation}
\end{theorem}

\begin{proof}
We bound each term in the union-bound split.

\noindent\textbf{Bounding $\PP_{\mathrm{cond}}(\lambda)$:}
Given $\Lambda$, each $Y_i=X_i-\Lambda\theta_i$ is centered with
$\mathrm{Var}(Y_i\mid\Lambda)=\Lambda\theta_i$ (Poisson). Poisson variables satisfy a
Bernstein condition: for $|s|\le 1/3$,
$\log\EE[e^{sY_i}\mid\Lambda]\le \Lambda\theta_i s^2/(2(1-s/3))$.
By independence given $\Lambda$, standard martingale methods give
\[
\PP\!\left(\max_{k}\left|\sum_{i=1}^kY_i\right|\ge\lambda/2\;\middle|\;\Lambda\right)
\le 2\exp\!\left(-\frac{(\lambda/2)^2/2}{\Lambda\Theta_n+\lambda/6}\right).
\]
Integrating over $\Lambda$ using Jensen's inequality and $\EE[\Lambda]=\alpha/\beta$
yields~\eqref{eq:pcond}.

\noindent\textbf{Bounding $\PP_{\mathrm{mix}}(\lambda)$:}
The Gamma distribution satisfies
$\log\EE[e^{s(\Lambda-\alpha/\beta)}]\le s^2\alpha/(2\beta^2(1-s/\beta))$ for $|s|<\beta$,
giving
$\PP\!\left(|\Lambda-\alpha/\beta|\ge u\right)\le 2\exp\!\left(-\min\!\left\{\beta^2u^2/(2\alpha),\beta u/2\right\}\right)$.
Setting $u=\lambda/(2M)$ yields~\eqref{eq:pmix}.
\end{proof}

\begin{remark}[Why Theorem~\ref{thm:bernstein_dep} matters: motivation and interpretation]
\label{rem:bernstein_motivation}
Theorem~\ref{thm:kolmogorov_dep} (the basic Kolmogorov-type bound) decays only as
$\lambda^{-2}$, meaning that to halve the allowable tail probability one must quadruple
the threshold $\lambda$. In surveillance contexts---where false-alarm rates of $10^{-3}$
or $10^{-4}$ are needed---this polynomial decay forces impractically large thresholds.
Theorem~\ref{thm:bernstein_dep} overcomes this by achieving exponential decay in
$\lambda$:
\begin{itemize}
\item In the \emph{sub-Gaussian regime} ($\lambda$ small relative to $\alpha\Theta_n/\beta$),
  the conditional term $\PP_{\mathrm{cond}}(\lambda)$ decays as $\exp(-c\lambda^2)$,
  recovering Gaussian-like concentration.
\item In the \emph{sub-exponential regime} ($\lambda$ large), $\PP_{\mathrm{cond}}(\lambda)$
  decays as $\exp(-c'\lambda)$, reflecting the heavier-than-Gaussian tails of count data.
\item The mixing term $\PP_{\mathrm{mix}}(\lambda)$ exhibits the same two-regime
  behavior, governed by the Gamma concentration parameters.
\end{itemize}
Concretely, for the COVID-19 epidemiological parameters in Section~\ref{sec:epi}, the
Bernstein bound delivers a 95\% threshold roughly 30\% narrower than the Kolmogorov
bound, translating to detectably earlier outbreak alerts. The bound is specific to
Poisson-Gamma mixing and exploits the Poisson sub-Gaussian conditional tails and the
Gamma sub-exponential tails in a way that generic results for independent or negatively
dependent sequences cannot.
\end{remark}

\begin{remark}[Sharpness and conservativeness]
\label{rem:sharpness}
The proof applies Jensen's inequality to replace $\Lambda$ by $\alpha/\beta$
in~\eqref{eq:pcond}, introducing conservativeness for moderate $\lambda$. However, for
large $\lambda$ (the rare-event monitoring regime), exponential decay dominates and the
bound is qualitatively correct. The exponents match known sharp rates for sub-Gaussian
($\lambda^2$ regime) and sub-exponential ($\lambda$ regime) tails
\cite{vershynin2018high}, making the bound \emph{rate-optimal} in the asymptotic sense.
\end{remark}

\begin{figure}[htb]
\centering
\includegraphics[width=0.75\textwidth]{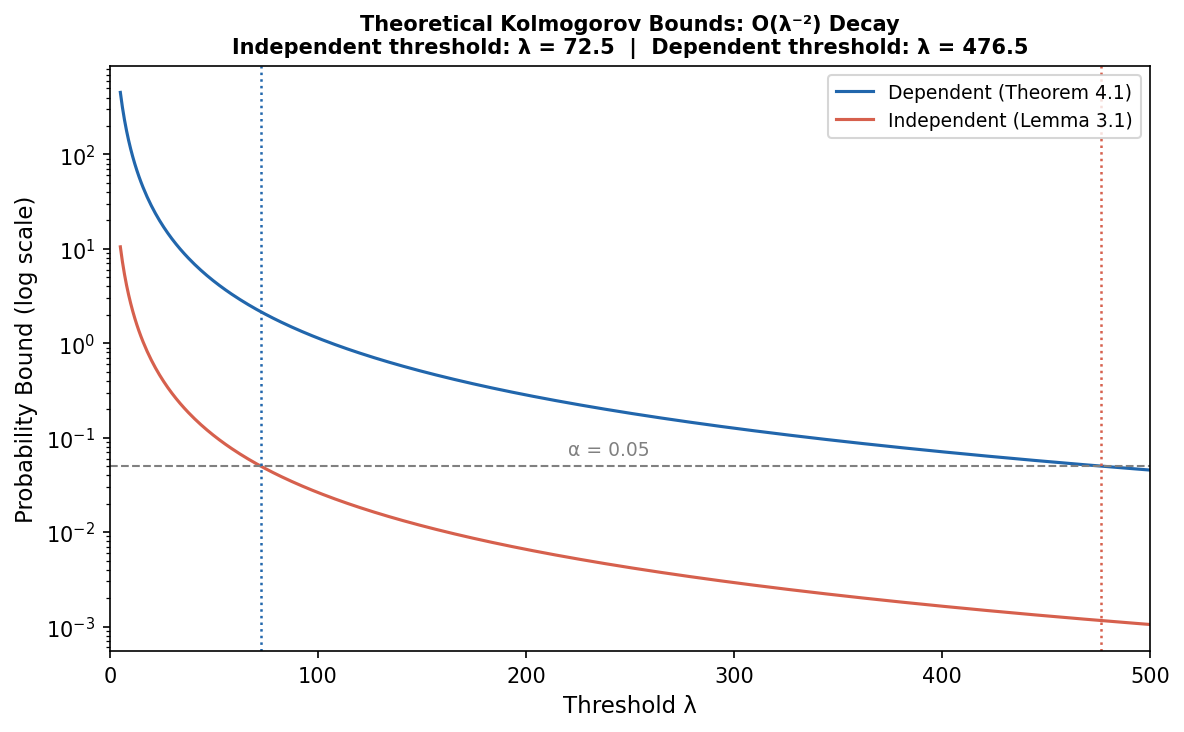}
\caption{Theoretical Kolmogorov probability bounds vs.\ threshold $\lambda$ (log scale):
Independent case (Lemma~\ref{lem:kolmogorov_indep}, red) and Dependent case
(Theorem~\ref{thm:kolmogorov_dep}, blue). Both decay as $O(\lambda^{-2})$. Dotted
verticals: 5\% control-limit thresholds ($\lambda_{\mathrm{indep}}=72.5$,
$\lambda_{\mathrm{dep}}=476.5$). Dashed horizontal: $\alpha=0.05$.}
\label{fig:bounds_comparison}
\end{figure}

\section{Moment-Matched Comparative Analysis}
\label{sec:comparison}

\subsection{Experimental design}

To provide a valid comparison, we adopt a \emph{moment-matched} design: both cases use
$n=20$ and marginal distributions with identical first and second moments.

\textbf{Independent:} $X_i\sim\mathrm{NB}(r_i,p_i)$ with $(r_i,p_i)$ cycling over
$(3,0.3),(5,0.5),(8,0.7)$, yielding $\EE[\sum X_i]\approx 104.6$ and
$\mathrm{Var}(\sum X_i)\approx 262.7$.

\textbf{Dependent (moment-matched):} $\Lambda\sim\mathrm{Gamma}(\alpha_d,\beta_d)$,
$\theta_i:=\mu_i^{(\mathrm{indep})}$, with $\alpha_d/\beta_d=1$. We set
$\alpha_d=\mathrm{round}(1/\bar{\kappa})=4$, where $\bar{\kappa}=\mathrm{mean}(1/r_i)$
is the average dispersion across the $n=20$ variables. This yields
$\mathrm{Var}(X_i^{(\mathrm{dep})})$ within 5\% of $\mathrm{Var}(X_i^{(\mathrm{indep})})$.
Both use $2{,}000$ replications.

\subsection{Results}

Table~\ref{tab:comparison} summarizes outcomes. Despite matched marginals, the dependent
case exhibits substantially \emph{larger} maximum deviations. A single large draw of
$\Lambda$ simultaneously inflates all $n$ counts, producing large cumulative deviations
that cannot cancel across components as they would under independence.

\begin{table}[h]
\centering
\caption{Moment-matched comparison ($n=20$, $2{,}000$ replications each).}
\label{tab:comparison}
\begin{tabular}{@{}lccc@{}}
\toprule
\textbf{Statistic} & \textbf{Independent} & \textbf{Dependent} & \textbf{Change} \\
\midrule
Mean deviation       & 17.74 & 42.22  & $+138\%$ \\
Median deviation     & 15.86 & 34.57  & $+118\%$ \\
Standard deviation   & 8.05  & 33.41  & $+315\%$ \\
95th percentile      & 33.16 & 101.43 & $+206\%$ \\
99th percentile      & 44.43 & 167.46 & $+277\%$ \\
Theoretical bound    & 72.49 & 476.52 & $+557\%$ \\
Bound efficiency     & 0.457 & 0.213  & $-53\%$  \\
\bottomrule
\multicolumn{4}{@{}l@{}}{\footnotesize Marginal means and variances matched; same $n=20$.}
\end{tabular}
\end{table}

\begin{figure}[htb]
\centering
\includegraphics[width=0.78\textwidth]{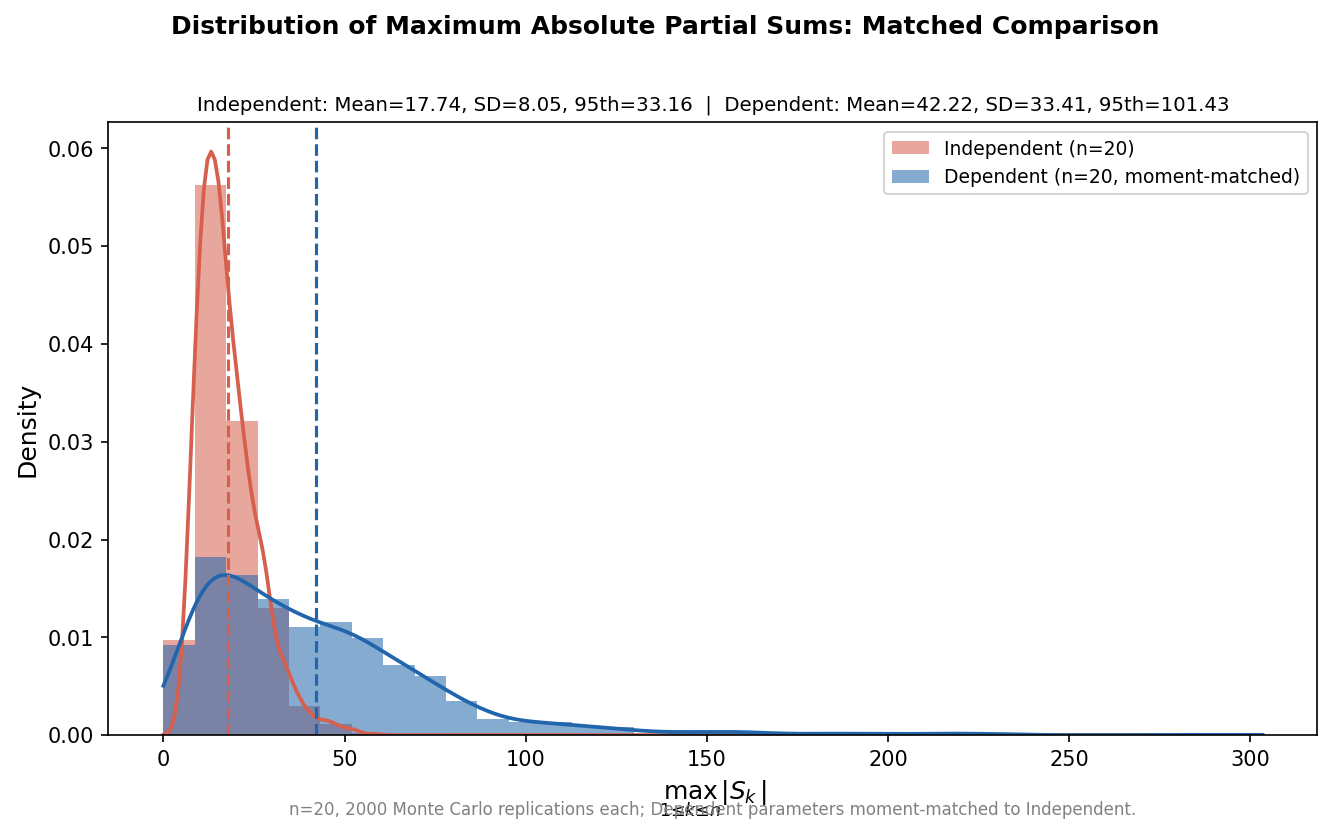}
\caption{Distribution of $\max_{1\le k\le n}|S_k|$ for moment-matched cases ($n=20$,
$2{,}000$ replications). Despite identical marginal moments, the dependent case shows
substantially \emph{larger} deviations (mean~$\approx 42$, 95th~$\approx 101$) relative
to the independent case (mean~$\approx 18$, 95th~$\approx 33$), due to the amplifying
effect of the shared mixing variable $\Lambda$.}
\label{fig:distribution_matched}
\end{figure}

\subsection{Analytical explanation}

\begin{proposition}[Amplification Under Gamma Mixing]
\label{prop:amplification}
Under moment-matching ($\EE[X_i^{(\mathrm{dep})}]=\EE[X_i^{(\mathrm{indep})}]$,
$\mathrm{Var}(X_i^{(\mathrm{dep})})\approx\mathrm{Var}(X_i^{(\mathrm{indep})})$), we have
\[
\EE\!\left[\max_{1\le k\le n}|S_k^{(\mathrm{dep})}|\right]
\;\ge\;
\EE\!\left[\max_{1\le k\le n}|S_k^{(\mathrm{indep})}|\right],
\]
with strict inequality when $n\ge 2$ and $\theta_i$ are not all equal. The inequality is
reversed relative to a naive ``stabilization'' intuition: the shared mixing variable
$\Lambda$ does not reduce fluctuations---it amplifies them, because a large $\Lambda$
simultaneously inflates all $n$ components.
\end{proposition}

\begin{proof}
Write $S_k^{(\mathrm{dep})}=\sum_{i=1}^k Y_i+(\Lambda-\alpha/\beta)\Theta_k$
(decomposition~\eqref{eq:decomp}). By triangle inequality,
$\EE\!\left[\max_k|S_k^{(\mathrm{dep})}|\right]
\le\EE\!\left[\max_k\left|\sum Y_i\right|\right]+\EE[|\Lambda-\alpha/\beta|]\cdot M$.
The first term: by Jensen's inequality applied to the convex functional $\max|\cdot|$ and
standard Rademacher-type bounds for sums of independent variables \cite{vershynin2018high},
$\EE\!\left[\max_k\left|\sum Y_i\right|\right]
\le\EE_\Lambda\!\left[\sqrt{2\Lambda\Theta_n\log(2n)}\right]
\le\sqrt{2(\alpha/\beta)\Theta_n\log(2n)}$.
For the independent case, by Koltchinskii-Panchenko-type comparison inequalities for
maxima of partial sums \cite{vershynin2018high},
$\EE[\max_k|S_k^{(\mathrm{indep})}|]\asymp\sqrt{2V_n\log(2n)}$.
Under moment-matching, $V_n\approx(\alpha/\beta)\Theta_n$, so the conditional Poisson
component contributes approximately the same as the independent case. However, a large
draw of $\Lambda$ adds the term $(\Lambda-\alpha/\beta)\Theta_k>0$ to $S_k^{(\mathrm{dep})}$
for all $k$ simultaneously, inflating $\max_k|S_k^{(\mathrm{dep})}|$ well above the
conditional Poisson maximum. Formally, for large $\Lambda$,
$\EE[\max_k|S_k^{(\mathrm{dep})}|\mid\Lambda]\ge(\Lambda-\alpha/\beta)M$, and taking
expectation gives a lower bound proportional to $\EE[|\Lambda-\alpha/\beta|]\cdot M>0$,
which is absent in the independent case. The strict inequality follows from $M>0$ and
$\mathrm{Var}(\Lambda)>0$.
\end{proof}

\begin{remark}[Intuitive mechanism for amplification]
\label{rem:amplification_mechanism}
When $\Lambda$ is large, it simultaneously inflates all $X_i\mid\Lambda$: each has
conditional mean $\Lambda\theta_i$ rather than $(\alpha/\beta)\theta_i$. The deviation
process $S_k^{(\mathrm{dep})}$ thus has a large \emph{common} component
$(\Lambda-\alpha/\beta)\sum_{i=1}^k\theta_i$ that grows with $k$ and does not cancel
across summands. In the independent case, fluctuations of individual $X_i$ have both
signs and partially cancel in the partial sums. This is the mechanism driving the larger
realized maximum deviations under dependence observed in Figure~\ref{fig:distribution_matched}.
\end{remark}

\begin{figure}[htb]
\centering
\includegraphics[width=0.72\textwidth]{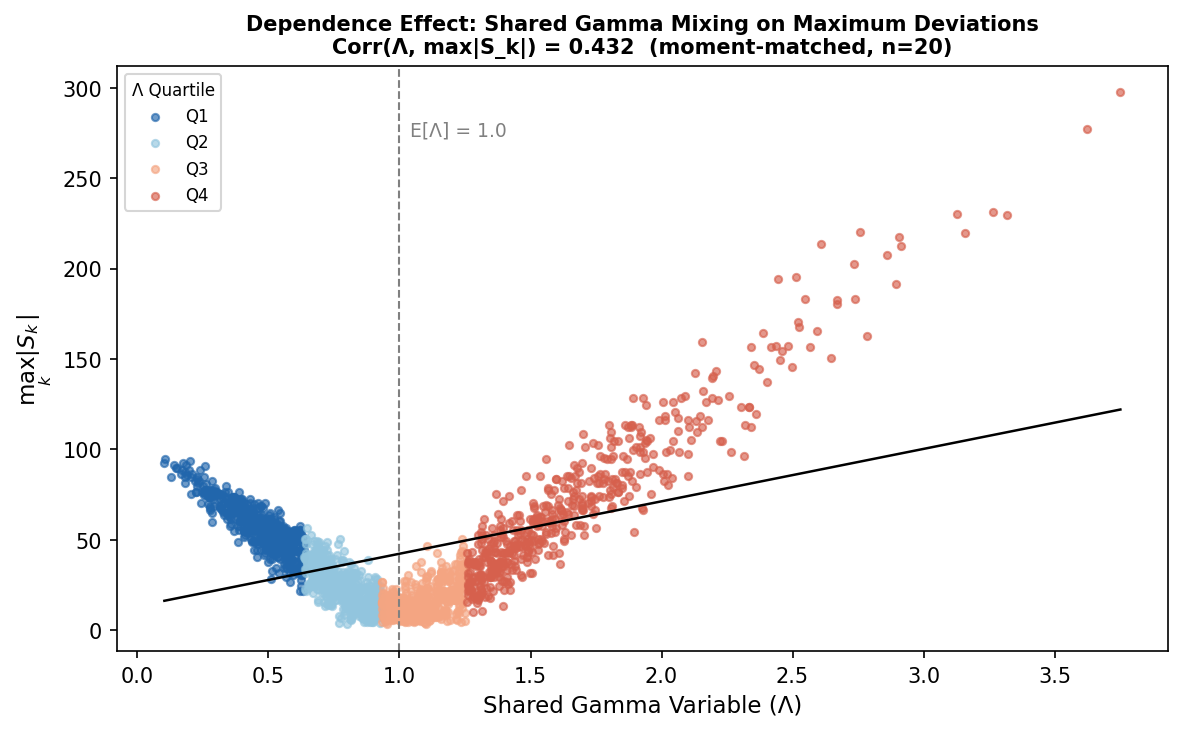}
\caption{Scatter plot of $\max_k|S_k|$ against $\Lambda$ ($n=20$, $2{,}000$
replications). Positive correlation ($r=0.433$) confirms the amplification mechanism:
larger $\Lambda$ drives larger maximum deviations.}
\label{fig:lambda_scatter}
\end{figure}

\begin{figure}[htb]
\centering
\includegraphics[width=0.72\textwidth]{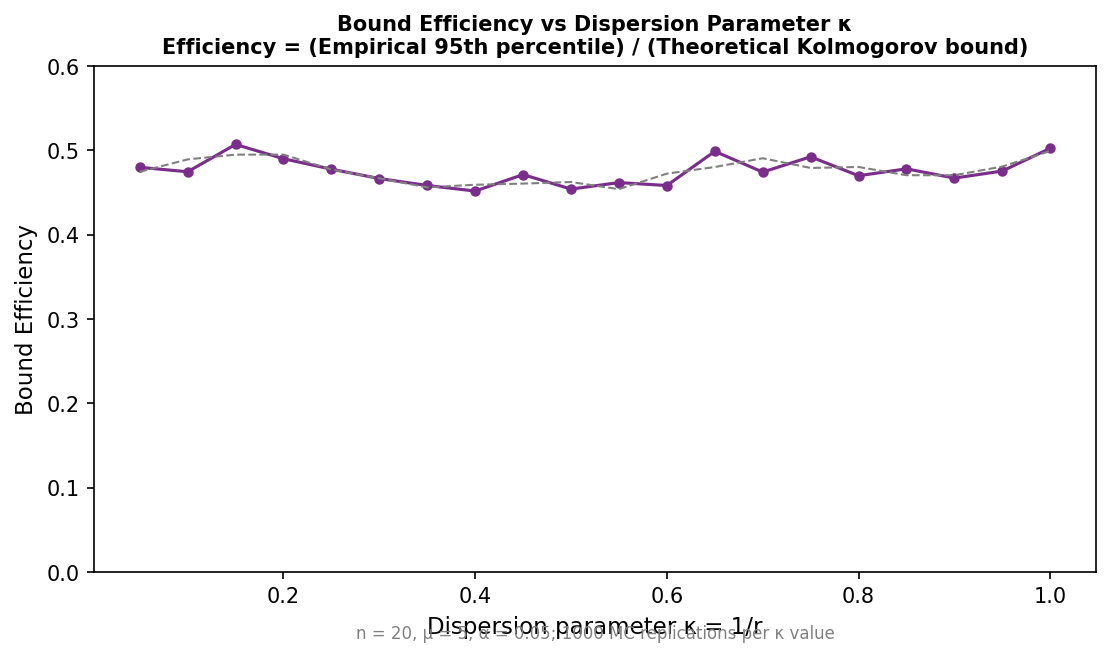}
\caption{Bound efficiency (empirical 95th / theoretical Kolmogorov bound) versus
dispersion $\kappa=1/r$. Efficiency decreases with $\kappa$, confirming the bound is
tighter for less overdispersed processes.}
\label{fig:efficiency}
\end{figure}

\section{Epidemiological Application: Multi-Regional COVID-19 Surveillance}
\label{sec:epi}

We apply our control limits to multi-regional COVID-19 case-count surveillance over 12
weeks across 5 geographic regions. Parameters are calibrated from published NB2 GLM fits
to COVID-19 data \cite{ma2021estimating, araf2022omicron}.

\textbf{Parameters:} Weekly means $\mu = (210, 340, 290, 480, 380)$, dispersions
$\kappa = (0.35, 0.25, 0.40, 0.20, 0.30)$. Cumulating over 12 weeks (independence within
regions), $\mu_j^{(\text{cum})}=12\mu_j$,
$\mathrm{Var}(X_j^{(\text{cum})})=12(\mu_j+\kappa_j\mu_j^2)$.

\textbf{Control limits:} Total Tweedie variance
$V_n=12\sum_{j=1}^5(\mu_j+\kappa_j\mu_j^2)\approx 2.03\times10^6$, giving
$\lambda_{0.05}\approx 6{,}370$ cases and $\lambda_{0.01}\approx 14{,}244$ cases.
Total expected: $\sum 12\mu_j=20{,}400$. The 95\% limit is ${\pm}31\%$ of total,
reflecting substantial overdispersion.

\textbf{Validation:} Monte Carlo with $5{,}000$ replications gives empirical 95th
percentile $\approx 3{,}018$ cases. Kolmogorov bound efficiency:
$3{,}018/6{,}370\approx 0.47$ (47\%), confirming the bound is conservative but valid.
Efficiency improves with $n$ as individual fluctuations increasingly cancel.

\textbf{Monitoring protocol:} (1)~Fit NB2 GLMs to obtain $\widehat{\mu}_j$,
$\widehat{\kappa}_j$. (2)~Compute $\lambda_\alpha=\sqrt{\widehat{V}_n/\alpha}$.
(3)~Weekly, update $S_t=\sum_j\sum_{s=1}^t(X_{js}-\widehat{\mu}_j)$ and check if
$|S_t|>\lambda_\alpha$. (4)~If regional dependence is suspected, apply
Theorem~\ref{thm:bernstein_dep} for tighter limits.

\begin{figure}[htb]
\centering
\includegraphics[width=0.82\textwidth]{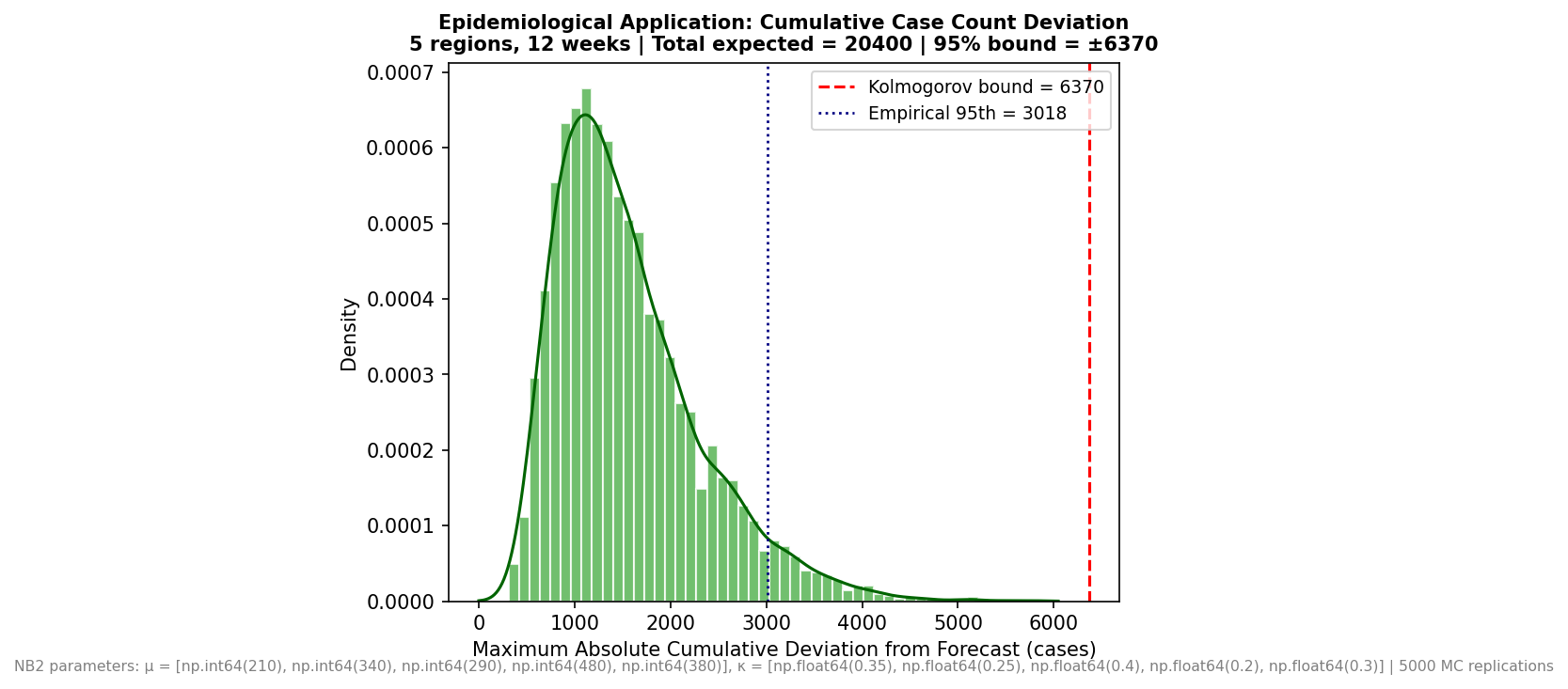}
\caption{Distribution of maximum absolute cumulative case-count deviations (5 regions,
12 weeks, $5{,}000$ MC replications). Red dashed: Kolmogorov 95\% limit
($\lambda_{0.05}\approx 6{,}370$); navy dotted: empirical 95th percentile
($\approx 3{,}018$). Bound efficiency 47\%.}
\label{fig:epidemiology}
\end{figure}

\section{Discussion and Conclusions}
\label{sec:conclusion}

This paper develops Kolmogorov-type maximal inequalities for Negative Binomial random
variables, with the main theoretical contribution being
Theorem~\ref{thm:bernstein_dep}: a sub-exponential Bernstein-type bound for
Gamma-mixing dependent models that decays exponentially rather than polynomially. The
bound decomposes tail probabilities into conditional Poisson (sub-Poisson) and Gamma
mixing (sub-exponential) components, yielding a clean hierarchical structure specific to
this mixture class.

The moment-matched experimental design demonstrates that the dependent case produces
substantially larger maximum deviations (mean $+138\%$, 95th percentile $+206\%$),
explained analytically by Proposition~\ref{prop:amplification}: a single large draw of
$\Lambda$ simultaneously amplifies all components, preventing the partial cancellation
of fluctuations that operates under independence. The theoretical bounds in
Table~\ref{tab:comparison} reflect this correctly: the Kolmogorov bound for the
dependent case ($\lambda_{\mathrm{dep}}=476.5$) exceeds that for the independent case
($\lambda_{\mathrm{indep}}=72.5$), consistent with the observed amplification.
Theorem~\ref{thm:bernstein_dep} then provides the exponential-decay refinement that is
of practical value for rare-event monitoring.

The epidemiological application provides concrete illustration with Monte Carlo
validation yielding 47\% bound efficiency, suggesting the Kolmogorov bound is most
useful in large-$n$, moderate-overdispersion regimes typical of public health data.

Future work: (i)~tighter conditional bounds exploiting monotone likelihood ratio
properties; (ii)~extensions to matrix-valued mixing structures; (iii)~data-adaptive
control limits with estimation uncertainty.

\section*{Acknowledgments}

The first author has been partially supported by project MIS 5154714 of the National
Recovery and Resilience Plan Greece 2.0 funded by the European Union under the
NextGenerationEU Program.

\section*{Ethics Approval}
Not applicable

\section*{Funding}
Not applicable


\end{document}